\documentclass[11pt]{amsart} 

\usepackage[utf8]{inputenc} 

\usepackage{geometry} 
\usepackage{graphicx} 

\usepackage{amsmath, amsthm}
\usepackage{booktabs} 
\usepackage{paralist} 
\usepackage{verbatim} 
\usepackage{subfig} 

\usepackage{draftcopy}

\usepackage{fancyhdr} 
\newtheorem{theorem}{Theorem}

\newtheorem{proposition}[theorem]{Proposition}
\newtheorem{example}[theorem]{Example}
\theoremstyle{definition}
\newtheorem{definition}[theorem]{Definition}

\numberwithin{equation}{section}

\def\perm{\pi}

\def\id{\iota}

\def\up#1{#1^\uparrow} 

\DeclareMathOperator{\sym}{S}

\DeclareMathOperator{\stat}{\ensuremath{DIS}}
\DeclareMathOperator{\stati}{\ensuremath{IDIS}}
\DeclareMathOperator{\den}{DEN}
\DeclareMathOperator{\hag}{HAG}
\DeclareMathOperator{\dhag}{DAG}
\DeclareMathOperator{\mad}{MAD}
\DeclareMathOperator{\maj}{MAJ}
\DeclareMathOperator{\mak}{MAK}
\DeclareMathOperator{\inv}{INV}
\DeclareMathOperator{\ska}{\ensuremath{eul}}
\DeclareMathOperator{\exc}{exc}
\DeclareMathOperator{\des}{des}
\DeclareMathOperator{\dbot}{Dbot}
\DeclareMathOperator{\ebot}{Ebot}
\DeclareMathOperator{\dtop}{Dtop}
\DeclareMathOperator{\etop}{Etop}
\DeclareMathOperator{\ddif}{Ddif}
\DeclareMathOperator{\edif}{Edif}
\DeclareMathOperator{\res}{Res}
\DeclareMathOperator{\ine}{Ine}
\DeclareMathOperator{\hagE}{L}

\title{An interesting new Mahonian permutation statistic}
\author{Mark C. Wilson}
\address{Department of Computer Science, University of Auckland,
Private Bag 92019 Auckland, New Zealand}
\email{mcw@cs.auckland.ac.nz}
\thanks{Thanks to Frank Ruskey, Mark Skandera,  Einar Steingr{\'{\i}}msson 
and Kyle Petersen for useful discussions.}

\begin{document}

\begin{abstract}
The standard algorithm for generating a random permutation gives rise to an 
obvious permutation statistic $\stat$ that is readily seen to be Mahonian. We 
give evidence showing that it is not equal to any previously published statistic.
Nor does its joint distribution with the standard Eulerian statistics $\des$ 
and $\exc$ appear to coincide with any known Euler-Mahonian pair.

A general construction of Skandera yields an Eulerian partner $\ska$ such that 
$(\ska, \stat)$ is equidistributed with $(\des, \maj)$. However $\ska$ 
itself appears not to be a known Eulerian statistic.

Several ideas for further research on this topic are listed.
\end{abstract} 

\maketitle

\section{The statistic}
\label{stat:intro}
\subsection{Random permutations}

The standard algorithm \cite[3.4.2, Algorithm P]{Knut1981b} for
uniformly generating a random permutation of $[n]:=\{1, \dots, n\}$  is as 
follows. Start with the identity permutation $\id = 1\dots n$ in the symmetric 
group $\sym_n$. There are $n$ steps labelled $n, n-1, \dots, 1$ (the last step 
can be omitted, but it makes our notation easier to include it here). At step $i$
a random position $j_i$ is chosen uniformly from $[i]$ and the current 
element in position $j_i$ is swapped with the element at position $i$. 
\begin{example}
The permutation $25413 \in \sym_5$ is formed by choosing 
$j_5 = 3, j_4 = 1, j_3 = 1, j_2 = 1, j_1 = 1$. Its inverse $41532$ is formed by 
choosing $j_5=2, j_4=3, j_3=2, j_2 = 1, j_1 = 1$.
\end{example}

In terms of multiplication in $\sym_n$, $\perm$ is a product of ``transpositions" 
$\prod_{i\leq n} (i j_i)$. Any of these 
``transpositions" may be the identity permutation. This representation as a 
``triangular product" gives a bijection between $\sym_n$ and 
the set of sequences $(j_1, \dots , j_n)$ that satisfy $1\leq j_i \leq i$ for 
all $i$.

Knuth attributes this algorithm to R. A. Fisher and F. Yates \cite{FiYa1938}, 
and a computer implementation was given by Durstenfeld \cite{Durs1964}. 
Recently \cite{Wils2009} the present author and others have 
studied the distribution of various quantities associated with the algorithm.

\subsection{The statistic}

For each $n\geq 1$, there is a map $\sym_n \to \sym_{n+1}$ that maps $\perm$ to the 
permutation $\up{\perm}$ that fixes $n+1$ and agrees on $1, \dots, n$ with $\perm$.
We let $\sym$ be the direct limit of sets induced by these maps. If we think of 
each $\up{}$ as an inclusion map, as is common, then $\sym$ is simply the union of 
all $\sym_n$. For our purposes a \emph{permutation statistic} is simply a function 
$T: \sym \to \mathbb{N}$.

Of course it is always possible to construct a statistic $T$ by for each $n$ 
making it equal to a given statistic $T_n$ on $\sym_n$. However unless the values of 
$T$ cohere for different values of $n$ this is not useful. We define a statistic 
on $\sym$ to be \emph{coherent} if it satisfies the following property. 
To be coherent, the identity $T(\perm) = T(\up{\perm})$ must hold for all $n$ 
and $\perm\in\sym_n$.  

We now define a (coherent) permutation statistic, which we denote by $\stat$, as follows. 
\begin{definition}
\label{def:stat}

At step $i$ of the algorithm described above, one symbol moves rightward a distance 
$d_i = i - j_i$ (possibly zero), and one symbol moves leftward the same distance. 
We define $\stat(\perm) = \sum_i d_i$, the total distance moved rightward by all
elements. 
\end{definition}

There is an alternative interpretation of $\stat$. The sequence of 
moves that formed $\perm$ starting from the identity will take 
$\perm^{-1}$ to the identity, and the moves are the same as selection sort. The 
algorithm then sorts $\perm^{-1}$ via selection sort. We can think of $\stat$ as 
a measure of the work done by selection sort when comparisons have zero cost. This 
model might be useful in analysing, for example, physical rearrangement of very 
heavy distinct objects.

In view of the last paragraph it makes sense also to consider the statistic 
$\stati$ given by $\stati(\perm) = \stat(\perm^{-1})$.

\begin{example}
\label{eg: compute stat}
For our running example $25413$, the value of $\stat$ is 
$(1+2+3+4+5)-(1+1+1+1+3) = 8$, while for $\stati$ the value is $6$. 
In terms of $\perm^{-1} = 41532$, the swaps used to create $\perm$ yield successively 
$41235, 31245, 21345, 12345$. 
\end{example}

Given a permutation for which we do not already know the $j_i$, we can find these
easily.

\begin{example}
\label{eg:find j}

Given $\perm = 25413$ as above, we can read off $j_5 = 3$ from $\perm$. 
Thus multiplying $\perm$ on the right by the transposition $(3 5)$ leads to 
$23415$. We have now reduced to $\perm = 2341$. We now read off $j_4 = 1$  
and reduce to $\perm = 231$. Continuing in this way we obtain 
$j_3 = 1, j_2 = 1, j_4 = 1$. 
\end{example}

At first sight it may appear that we must search to find the position of 
symbol $i$ at step $i$, leading to a quadratic time algorithm for the procedure 
of the last example. However this is not the case, provided we compute $\stat$ 
and $\stati$ simultaneously, and the entire computation can be done in linear 
time (note that computing $\perm^{-1}$ from $\perm$ is a linear time operation). 
Note that, for example, it is still unknown whether the number of inversions $\inv$ 
of a permutation can be computed in linear time.

\begin{example}
\label{eg: linear time}
In the running example $\perm = 25413, \perm^{-1} = 41352$, we read off 
$j_5(\perm) = 3, j_5(\perm^{-1}) = 2$. To multiply $\perm$ on the right by the 
transposition $(3 5)$ we need not scan all of $\perm$, because we know the location
of the symbol $5$, namely $j_5(\perm^{-1})$. Thus the multiplication takes constant 
time. We can either multiply $\perm^{-1}$ on the left by $(3 5)$ or on the right by 
$(2 5)$. Each leads to the same answer, namely $41235$, and this is the inverse of 
the updated $\perm$. Continuing in this way we obtain the result of the last 
example.
\end{example}

It will be helpful to know the values of $\stat$ on some special permutations. 

\begin{example}
We define
\begin{align*}
\perm_0 & = n(n-1)\dots 1\\
\perm_1 & = 2\dots n1\\
\perm_1^{-1} & = n12\dots (n-1)\\
\end{align*}

Note that $\perm_0$ is created by the algorithm by choosing $j_i = n+1-i$ 
provided $n+1 - i < i$, whereupon all later swaps are trivial. Also $\perm_1$ 
is created by choosing $j_i = 1$ for all $i$, while $\perm_1^{-1}$ is formed 
by choosing $j_i = i - 1$ for $i\geq 2$. Thus 
\begin{align*}
\stat(\id) & = 0\\
\stat(\perm_0) & = \lfloor n^2/4 \rfloor = 
\begin{cases}
\frac{n^2}{4} &  \text{ if $n$ is even;} \\
\frac{n^2 - 1}{4} & \text{ if $n$ is odd}.
\end{cases}\\
\stat(\perm_1) & = n(n-1)/2 \\
\stat(\perm_1^{-1}) & = n - 1\\
\end{align*}

\end{example}
The maximum value of $\stat$ on $\sym_n$ is $n(n-1)/2$, 
corresponding uniquely to the $n$-cycle $\perm_1$. The minimum 
value of $\stat$ on $\sym_n$ is $0$, corresponding uniquely to the identity $\id$. 

As a random variable, the restriction $\stat_n$ of $\stat$ to $\sym_n$ 
is the sum of $\stat_{n-1}$ and a
random variable $U_n$ that is uniform on $[0..n-1]$. Thus, iterating this recurrence, 
we see that $\stat_n$  has probability generating
function $F_n(q):=\prod_{i=1}^n \frac{1 - q^i}{1 - q}$. 
This is the definition of a \emph{Mahonian statistic} on $\sym_n$. 
Note that $n(n-1)/2 - \stat = \sum_i j_i$ is also Mahonian by the symmetry of the 
Mahonian distribution.

\section{$\stat$ is not trivially equal to a known statistic}
\label{sec:new}

\begin{table}
\caption{Values of some permutation statistics for $n=4$.}
$$
\begin {matrix}
\perm& \stat & \inv & \maj & \den & \mad & \mak & \hag\\
1234&0&0&0&0&0&0&0\\
1243&1&1&3&3&1&3&3\\
1324&1&1&2&2&1&2&2\\
1342&3&2&3&5&1&4&5\\
1423&2&2&2&2&2&2&3\\
1432&2&3&5&3&2&5&2\\
2134&1&1&1&1&1&1&1\\
2143&2&2&4&4&2&4&4\\
2314&3&2&2&3&1&3&3\\
2341&6&3&3&6&1&5&6\\
2413&4&3&2&3&2&3&4\\
2431&4&4&5&4&2&6&3\\
3124&2&2&1&1&2&1&2\\
3142&4&3&4&4&3&5&5\\
3214&2&3&3&2&2&3&1\\
3241&5&4&4&5&2&6&4\\
3412&4&4&2&3&2&3&4\\
3421&5&5&5&4&2&6&3\\
4123&3&3&1&1&3&1&3\\
4132&3&4&4&2&4&4&2\\
4213&3&4&3&2&3&3&2\\
4231&3&5&4&3&3&5&1\\
4312&5&5&3&4&3&3&5\\
4321&4&6&6&5&3&6&4
\end {matrix}
$$
\label{t:compare}
\end{table}

Tabulating numerical values makes it clear that $\stat$ is not equal to any of 
the most well-known Mahonian statistics. Table~\ref{t:compare} gives the values 
of $\stat$ and several other Mahonian statistics when $n=4$ (it is amusing to note 
that they all coincide on the element $2134$ - the obvious conjecture that they 
always coincide on $2134\dots n$ is in fact correct). These statistics 
are $\inv, \maj, \den, \mad, \mak, \hag$. We recall the unified definition of 
these statistics given in \cite{CSZ1997a}. We first require some partial 
statistics.

\begin{definition}
A \emph{descent} is an occurrence of the event $\perm(i) > \perm(i+1)$. The 
index $i$ is the \emph{descent bottom} and $\perm(i)$ is the corresponding 
\emph{descent top}. 
\end{definition}

Each $\perm$ can be uniquely decomposed into \emph{descent blocks} 
(maximal descending subwords). Denote the first and last letter of each 
block of length at least 2 by $c(B), o(B)$.  The 
\emph{right embracing number} of $\perm(i)$ is the number of that are 
descent blocks strictly to the right of the block containing $\pi(i)$ and for 
which  $c(B) > \pi(i) > o(B)$. The sum of all right embracing numbers is denoted 
by $\res(\perm)$.

\begin{example}
For $\perm_1$ the descent blocks are all of length $1$ except for the last one, 
$n1$. The right embracing number of each letter $2, \dots, n-1$ is $1$ and the 
right embracing number of $n$ and of $1$ are each $0$. For $\perm_1^{-1}$ 
there is again a single nontrivial descent block, namely $n1$, and all right 
embracing numbers are $0$. For $\perm_0$ there is a single descent block of 
length $n$ and all right embracing numbers are $0$.
\end{example}

\begin{definition}
An \emph{excedance} is an occurrence of the event 
$\perm(i) > i$. The index $i$ is the \emph{excedance bottom} and $\perm(i)$ 
is the corresponding \emph{excedance top}. The sum of all descent/excedance 
tops/bottoms of $\perm$ we denote by $\dtop(\perm), \etop(\perm), \dbot(\perm), 
\ebot{\perm}$. The differences $\ddif(\perm)$ and $\edif(\perm)$ are given by 
$\ddif = \dtop - \dbot, \edif = \etop - \ebot$.
\end{definition}

There is a unique decomposition $\perm$ into $\perm_E$ and $\perm_N$, where 
$\perm_E$ is the subsequence formed by excedances and $\perm_N$ the subsequence formed 
by nonexcedances. For our running example $\perm = 25413$, we have 
$\perm_E = 25$ and $\perm_N = 413$. For the inverse $41532$ we have respectively
$45$ and $132$. We define $\ine(\perm) = \inv(\perm_E) + \inv(\perm_N)$.

For each excedance bottom $i$ we define $\hagE(i)$ to be the number of indices 
$k$ such that $k < i$ and $a_k \leq i$; let $\hagE$ be the sum over all such $i$.

\begin{example}
Note that $(\perm_1)_E = 23\dots n$ and $(\perm_1)_N = 1$. Similarly 
$(\perm_1^{-1})_E = n$ and $(\perm_1^{-1})_N = 12\dots (n-1)$.  Also $(\perm_0)_E = n(n-1) \dots t+1$ and 
$(\perm_0)_N = t\dots 1$, where $t = \lfloor n/2 \rfloor$. 
\end{example}

The values of the partial statistics defined above are tabulated in 
Table~\ref{t:partial}.

\begin{proposition}[\cite{CSZ1997a}]
We have
\begin{align*}
\mak & = \dbot + \res \\
\mad & = \ddif + \res \\
\den & = \ebot  + \ine \\
\inv & = \edif + \ine \\
\hag & = \edif + \inv(\perm_E) - \inv(\perm_N) + \hagE\\
\end{align*}
In addition $\maj$ is the sum of indices corresponding to descent tops.
\end{proposition}

\begin{table}
\caption{Values of partial statistics on special permutations 
($t = \lfloor n/2 \rfloor$)}
\begin{tabular}{|c|c|c|c|c|c|c|c|}
\hline
perm & $\ebot$ & $\edif$ & $\dbot$ & $\ddif$ & $\res$ & $\ine$ & $\hagE$ \\ \hline
$\perm_0$ &$t(t+1)/2 $ & $\lfloor n^2/4 \rfloor$  & $n(n-1)/2$ & $n - 1$ 
& $0$ & $\lfloor (n-1)^2/4 \rfloor$ & $0$ \\ \hline
$\perm_1$ & $n(n-1)/2$ & $n-1$ & $1$ & $n-1$ & $n-2$ & $0$ & $(n-1)(n-2)/2$ \\ \hline
$\perm_1^{-1}$ & 1 & $n-1$ & $1$ & $n-1$ & $0$ & $0$ & $0$\\ \hline
\end{tabular}
\label{t:partial}
\end{table}

\subsection{Trivial bijections}
\label{ss:trivial}

To show that statistics $T$ and $T'$ are different, it suffices to find some 
$n$ and some $\perm\in \sym_n$ for which $T(\perm) \neq T'(\perm)$. However it 
may be the case that $T$ and $T'$ agree  on $\sym_n$ for some larger values of 
$n$. If both $T$ and $T'$ are coherent, this possibility cannot occur.

Note that $\stat$ and $\stati$, along with all statistics from previous 
literature with which we compare them here, are coherent. Thus simply computing 
values for small $n$, as in the previous section, is usually enough to 
distinguish the statistics. However we can often give a general construction of 
permutations for which a given pair of statistics differs greatly. 

Although $\stat$ does not equal any of the well-known statistics of the previous 
section, is possible \emph{a priori} that $\stat$ has the form $S\circ g$ 
where $S$ is a known Mahonian statistic and $g$ is a filtered bijection of 
$\sym$ (a bijection of $\sym$ that bijectively takes $\sym_n$ to $\sym_n$ for 
each $n$).

In this section we consider the so-called ``trivial" involutions of $S_n$ 
(there is a nontrivial bijection $\Phi$ of $\sym_n$ introduced in \cite{CSZ1997a}; 
we give more details in Section~\ref{ss:eul-mah}.)  These involutions are
inversion (group-theoretic inverse), reversal (reverse the order of the letters) 
and complementation (subtract each letter from $n+1$). Then in the obvious 
notation $R$ and $C$ commute and $IR = CI, IC = RI$. Thus $I, R, C$ generate a 
group $G$ isomorphic to the dihedral group of order $8$.

For example we have
\begin{align*}
(25413)^{I} & = 41532 \\
(25413)^R & = 31452 \\
(25413)^C & = 41253 \\
(25413)^{IC} & = 25134 \\
(25413)^{IR} & = 23514 \\
(25413)^{RC} & = 35214 \\
(25413)^{IRC} & = 41532  \\
\end{align*}

We shall show that $\stat$ is not trivially equivalent to any well-known 
statistic. In the absence of a standardized database of permutation statistics, 
we define ``well-known" to mean ``mentioned in at least one of the papers 
\cite{CSZ1997a, BaSt2000}". We define $\Sigma$ to be the set consisting of 
well-known Mahonian statistics.

In \cite{BaSt2000} it is shown how all known ``descent-based" Mahonian statistics 
can be written in terms of ``Mahonian $d$-functions" for some $d\leq 4$. Each 
such $d$-function simply computes the numbers of occurrences of a certain 
generalized permutation pattern of length at most $d$, then sums this 
process over a finite number of such patterns. In particular in Table 1 of the 
above article, all 14 Mahonian 3-functions (up to trivial bijections) are given.
In \cite{CSZ1997a} the images of these statistics under a bijection $\Phi$ were 
also considered. We consider this bijection in Section~\ref{ss:eul-mah}.

\begin{theorem}
\label{t:new}
There do not exist $S\in \Sigma$ and $g\in G$ such that $\stat = S\circ g$.
\end{theorem}

\begin{proof}
Note that $\perm_0^R = \id = \perm_0^C$ while  $\perm_0$ is a product of 
$\lfloor n/2 \rfloor$ disjoint transpositions, and hence $\perm_0^I = \perm_0$. 
Hence the orbit of $\perm_0$ under $G$ is the set $\{\id, \perm_0\}$ and this 
is also the orbit of $\id$. The orbit of $\perm_1$ under $G$ is disjoint from 
that of $\perm_0$ and $\id$. It consists of $\perm_1, \perm_1^{-1}, \perm_1^R = 
\perm_1^C = 1n\dots 2, \perm_1^{-R} = (n-1)\dots 1n$. 

It follows that if $S\circ g = \stat$ for some permutation statistic $S$ and 
element $g\in G$, then $S(\perm_0)$ must equal zero or $\stat(\perm_0)$. 
However it is readily seen by comparing with Table~\ref{t:partial} that none of 
the statistics in \cite{CSZ1997a} satisfy this property. This includes those 
mentioned in passing, such as LAG and SIST.

Now consider the statistics in \cite[Table 1]{BaSt2000}, given in terms of 
permutation pattern counts. Any pattern that 
is not strictly descending does not occur in $\perm_0$, so we need only count 
occurrences of $ba, cba, cb-a, c-ba$. Again, none of these lead to zero or 
$\stat(\perm_0)$, since the number of occurrences of these four patterns in 
$\perm_0$ is respectively $n-1, n-2, (n-1)(n-2)/2, (n-1)(n-2)/2$. 

Finally we consider Haglund's statistic $\hag$ and a descent-based variant 
$\dhag$ as defined in \cite{BaSt2000}. The statistic $\dhag$ can be dealt with 
by counting pattern occurrences in $\perm_0$ as above. However it is not as 
easy to differentiate $\hag$ from $\stat$ by using our special permutations. 
In fact when $n$ is even, $\hag$ and $\stat$ take the same value on $\perm_0$ (they 
coincide with $\edif$). When $n$ is odd, $\hag$ is smaller than $\stat$ by 
$(n-1)/2$. We instead use the permutation $\perm_2 = n2\dots (n-1)1$ formed from $\id$ 
by a single transposition. Its orbit under $G$ consists of itself and its reverse
$1(n-1)\dots 2n$, and $\stat$ takes the values $n-1$ and $\lceil (n-2)^2/4\rceil$ 
respectively on these two elements. However, $\hag(\perm_2) = 1$.

\end{proof}

\subsection{Euler-Mahonian pairs and nontrivial bijections}
\label{ss:eul-mah}

In \cite{CSZ1997a} a bijection $\Phi$ of $\sym_n$ was given and it was shown that 
$\Phi$ had appeared (somewhat disguised) in several previous papers. The key 
property of $\Phi$ is that it takes $(\des, \dbot, \ddif, \res)$ to 
$(\exc, \ebot, \edif, \ine)$. This then gives access to equidistribution results 
for Euler-Mahonian pairs. The term \emph{Euler-Mahonian} refers in the literature 
to a bistatistic $(e, M)$ such that $e$ is Eulerian, $M$ is Mahonian, and the 
joint distribution of $(e, M)$ is the same as that of another well-known pair 
$(e', M')$. Originally the term was used only for $(e', M') = (\des, \maj)$.  
Other authors, for example \cite{BaSt2000, CSZ1997a} 
allow more possibilities for $(e', M')$, and aim to classify these bistatistics 
up to equidistribution.

In \cite[Table 2]{BaSt2000} seven equivalence classes (under equidistribution)
of Euler-Mahonian pairs $(\des, T)$ were given for $n = 5$ (note that the 
second matrix, corresponding to $\maj$, has an error: in the row indexed by 
$\des = 2$, the entries listed as 14 should be 16). This corresponds to 
14 Mahonian statistics $T$. It is easy to see that $\stat$ does not occur in 
this table, because its maximum value occurs on $\perm_1$ and $\des(\perm_1) = 1$, 
yet none of the seven distributions has a nonzero entry in the $(1, 10)$ position.
We can also check easily that $(\exc, \stat)$ has a different distribution from all the 
entries in the table. Thus if $T'$ is the image of such a $T$ under $\Phi$, 
then $T' \neq \stat$.

\begin{table}
$$
\begin {array}{ccccccccccc} 
1&0&0&0&0&0&0&0&0&0&0\\
0&4&6&8&8&0&0&0&0&0&0\\
0&0&3&7&10&22&15&9&0&0&0\\
0&0&0&0&2&0&5&6&9&4&0\\
0&0&0&0&0&0&0&0&0&0&1
\end {array}
$$
\newline
$$
\begin {array}{ccccccccccc} 
1&0&0&0&0&0&0&0&0&0&0\\
0&4&3&5&3&3&3&2&1&1&1\\
0&0&6&6&13&12&9&9&8&3&0\\
0&0&0&4&3&7&8&4&0&0&0\\
0&0&0&0&1&0&0&0&0&0&0
\end {array}
$$
\label{t:joint}
\caption{Joint distributions $(\exc, \stat)$ and $(\exc, \stati)$ for $n=5$.}
\end{table}

\begin{table}
$$
\begin {array}{ccccccccccc}
1&0&0&0&0&0&0&0&0&0&0\\
0&4&3&5&5&2&3&2&1&0&1\\
0&0&6&8&12&14&11&7&5&3&0\\
0&0&0&2&3&6&5&6&3&1&0\\
0&0&0&0&0&0&1&0&0&0&0
\end {array} 
$$
\newline
$$
\begin {array}{ccccccccccc} 
1&0&0&0&0&0&0&0&0&0&0\\
0&4&3&5&5&2&2&3&1&0&1\\
0&0&6&8&12&15&11&6&5&3&0\\
0&0&0&2&3&5&6&6&3&1&0\\
0&0&0&0&0&0&1&0&0&0&0
\end {array}
$$
\label{t:joint}
\caption{Joint distributions $(\des, \stat)$ and $(\des, \stati)$ for $n=5$.}
\end{table}

We still need to check $\hag$. By direct computation we can show readily that 
\begin{align*}
\Phi(\perm_0) & = \perm_1 \\
\Phi(\perm_1) & = (\perm_1)^{-1} \\
\end{align*}

Suppose that $T = \stat \circ \Phi$ for some $T\in \Sigma$. Then $T(\perm_0) 
= n(n-1)/2$ and $T(\perm_1) = n - 1$. Clearly $\hag$ fails this test. If 
$T\circ \Phi = \stat$ then $T(\perm_1) = \lceil n^2/4 \rceil$ which again $\hag$ 
fails.

\section{An Eulerian partner for $\stat$}

Skandera \cite{Skan2002} gave a general procedure for associating to each Mahonian 
statistic $M$ another statistic $e$ that is Eulerian and such that the pair 
$(e, M)$ is Euler-Mahonian (equidistributed with $(\des, \maj)$). Of course, 
such an Eulerian statistic may not be known or particularly interesting.

Applying this procedure to $\stat$ yields an Eulerian statistic $\ska$. 
Concretely, $\ska(\perm)$ is obtained from the numbers $d_i$ by listing them in 
order, and counting each time we encounter a number larger than the current 
record (the record being initialized to zero). For example, for our running 
example $25413$ we have $d = (0,1,2,3,2)$ and so $\ska$ takes the value $3$. 
Also note that $\ska(\perm_0) = \lceil (n-1)/2 \rceil$ while $\ska(\id) = 0$ 
and $\ska(\perm_1) = n - 1$.

A well-known Eulerian statistic is the number of excedances $\exc$. 
Now $\exc$ agrees with $\ska$ on $\perm_0$ and $\perm_1$. Also,
$\ska$ and $\exc$ are equal when $n = 3$. Nevertheless, $\ska$ is not equal to 
$\exc$ in general, nor does it equal $\des$.

Eulerian statistics in the literature are less commonly found than Mahonian ones. 
As far as I am aware, $\ska$ is itself new, but this is based on much less evidence than 
the corresponding claim about $\stat$.

\section{Further comments} 
\label{sec:extn}

The current paper gives substantial evidence that the statistic $\stat$ is really new.
In order to check thoroughly whether a permutation statistic is new to the literature, one would 
ideally check a database of such statistics. I have not found such a database. I 
propose that as a minimum, tables of values for $n=4$, along with the joint distribution with 
$\exc$ and $\des$ for $n=5$, be included in all papers dealing with this topic, 
to allow easy comparison. It would then be much easier to show that 
the entire group $\Gamma$ generated by $G$ and $\Phi$ does 
not have any element $g$ with $T\circ g = \stat$ for some known Mahonian 
$T$, since all such $T$ of which I am aware are consistent. 

It may be desirable to find a ``static" description of $\stat$ and $\stati$, 
which have been defined ``dynamically". I do not know a systematic way to do this
 (one possible idea is to find linear combinations of the above partial statistics 
 that fit the values for small $n$).  A related question is to determine whether 
$\stat$ can be written as a Mahonian $d$-function for some $d$.

The statistic $\stati$ should extend to words via the selection sort interpretation. 
Whether this statistic is Mahonian on words should be investigated and I intend to do this in 
future work.

Note: As I was preparing this article I was made aware of completely independent recent work by 
T. K. Petersen \cite{Pete2010} that also discusses the statistic $\stat$ and some generalizations. 
The intersection between the topics of these two papers is small, and the reader should consult 
both articles for a fuller picture.

\bibliographystyle{amsalpha}
\bibliography{mahonian}

\end{document}